\documentclass[reqno]{amsart}
\usepackage{graphicx}
\usepackage{trajan}
\usepackage{multirow}
\usepackage{latexsym,array,delarray,amsthm,amssymb,epsfig}
\usepackage{amsmath}
\usepackage{yhmath}
\usepackage{mathdots}
\usepackage{MnSymbol}
\usepackage{tikz}
\usepackage{adjustbox}
\usepackage{rotating}
\usepackage{longtable}
\usepackage{enumitem}
\usepackage{caption}
\usepackage{wasysym}
\usetikzlibrary{shapes.misc,shadows}
\theoremstyle{plain}
\newtheorem{thm}{Theorem}[section]
\newtheorem{lemma}[thm]{Lemma}

\newtheorem{cor}[thm]{Corollary}

\newtheorem*{thm*}{Theorem}
\newtheorem*{lemma*}{Lemma}
\newtheorem*{prop*}{Proposition}
\newtheorem*{cor*}{Corollary}
\newtheorem*{conj*}{Conjecture}

\theoremstyle{definition}

\newtheorem{ex}[thm]{Example}

\theoremstyle{remark}
\newtheorem*{rmk}{Remark}

\newcommand{\A}{\mathcal{A}}

\newcommand{\g}{\mathfrak{g}}

\begin{document}
\date{}

\title{Subinvariance in Leibniz Algebras}
\author{Kailash C. Misra, Ernie Stitzinger and Xingjian Yu}
\address{Department of Mathematics, North Carolina State University, Raleigh, NC 27695-8205}
\email{misra@ncsu.edu, stitz@ncsu.edu, xyu13@ncsu.edu }
\thanks{KCM is partially supported by Simons Foundation grant \#  636482}

\begin{abstract}
Leibniz algebras are certain generalizations of Lie algebras. Motivated by the concept of subinvariance in group theory, Schenkman studied properties of subinvariant subalgebras of a Lie algebra. In this paper we define subinvariant subalgebras of Leibniz algebras and study their properties. It is shown that the signature results on subinvariance in Lie algebras have analogs for Leibniz algebras.

\end{abstract}

\maketitle
\bigskip
\section{Introduction}

Let $\A$ be a vector space over field $F$ equipped with a bilinear multiplication map $ \A \times \A \longrightarrow \A$ given by $(x, y) \longrightarrow xy$ for all $x, y \in \A$. For each $x \in \A$, define the left (resp. right) multiplication operator $L_x :\A \longrightarrow \A$ (resp. $R_x :\A \longrightarrow \A$ defined by $L_x(y) = xy$ (resp. $R_x(y) = yx$) for all $y \in \A$. The vector space $\A$ is a left (resp. right) Leibniz algebra if $L_x$ (resp. $R_x$) is a derivation for each $x \in \A$. Following Barnes \cite{B}, in this paper Leibniz algebras will always refer to finite dimensional left Leibniz algebras. Observe that a Lie algebra is a Leibniz algebra but not conversely. In particular the Leibniz algebra $\A$ has an abelian ideal ${\rm Leib}(\A)={\rm span}\{x^2= xx \mid x\in \A\}$. The ideal Leib$(\A)= \{0\}$ if and only if $\A$ is a Lie algebra. For a subalgebra $B$ of $\A$, the left (resp. right) centralizer of $B$ in $\A$ is defined by $Z^{\ell}_{\A}(B) = \{x \in \A \mid xa = 0 \ \ {\rm for} \, {\rm all} \ \ a \in B\}$ (resp. $Z^r_{\A}(B) = \{x \in \A \mid ax = 0 \ \  {\rm for} \, {\rm all} \ \ a \in B\}$). The cetralizer of $B$ in $\A$ is $Z_{\A}(B) = Z^{\ell}_{\A}(B) \cap Z^r_{\A}(B)$. The left (resp. right) center of $\A$ is $Z^{\ell}(\A) = Z^{\ell}_{\A}(\A)$. (resp. $Z^r(\A) = Z^r_{\A}(\A)$) and the center of $\A$ is $Z(\A) = Z^{\ell}(\A) \cap Z^r(\A)$. Note that ${\rm Leib}(\A) \subseteq Z^{\ell}(\A)$ and the center of $\A$ is $Z(\A) = Z^{\ell}(\A) \cap Z^r(\A)$. We denote $L(\A) = \{L_x \mid x \in\A\}$ which is a Lie algebra of linear transformations on $\A$. For a Leibniz algebra $\A$, we define the ideals $\A^{(1)} = \A = \A^{1}$, $\A^{(i)} = \A^{(i-1)} \A^{(i-1)}$ and $\A^{i} = \A\A^{i-1}$ for $i \in \mathbb{Z}_{\geq 2}$. The Leibniz algebra is said to be solvable (resp. nilpotent) if $\A^{(m)} = \{0\}$ (resp. $\A^{m} = \{0\}$) for some positive integer $m$. The maximal solvable (resp. nilpotent) ideal of $\A$ is called the radical (resp. nilradical) and denoted by rad$(\A)$ (resp. nilrad$(\A)$).

The concept of subinvariance was popularized for groups by Wielandt \cite{W} and for Lie algebras by Schenkman \cite{SE}. A subalgebra $B$ of the Leibniz algebra $\A$ is subinvariant in $\A$ if there exists a chain of subalgebras $B = B_t \subset B_{t-1} \subset \cdots \subset B_0 =A$ such that $B_i \vartriangleleft B_{i-1}$ for all $1 \leq i \leq t$. Such a series is called a normal series from $B$ to $A$. This normal series is a composition series from $B$ to $A$ if $B_i$ is a maximal ideal of $B_{i-1}$, $1 \leq i \leq t$ and in such case $B_{i-1}/B_i$ are called composition factors. As shown in  \cite{DMS}, the Leibniz algebra $\A$ is nilpotent if and only if all subalgebras are subinvariant in $\A$. In this paper we show that most important results on subinvariance of Lie algebras in \cite{SE} has Leibniz algebra analogs although techniques to obtain these results are different. 

We organize this paper as follows. In Section 2, we give results on subinvariant subalgebras of Leibniz algebras over any field $F$. We define a function $f$ on a pair of subsets of a Leibniz algebra $\A$ and show that it has certain nice property which is used to prove several results through out the paper. Let $B$ be a subinvariant subalgebra of $\A$ and $B^{\omega}$ be the smallest member in the lower central series of $B$, then it is shown that $B^{\omega}$ is an ideal of $\A$. Also if $Z_{\A}(B) = \{0\}$, we show that $Z_{\A}(B^{\omega}) \subseteq B^{\omega}$. In Section 3, we collect the results on subinvariant subalgebras of Leibniz algebras over field $F$ of characteristic zero and show that the important results in this section fail over field of positive characteristic. We show that nilpotent (resp. solvable) subinvariant subalgebra of $\A$ is contained in $nilrad(\A)$ (resp. $rad(\A)$) and for any subinvariant subalgebra $B$ of $\A$, we show that $nilrad(B) = B \cap nilrad(\A)$ and 
$rad(B) = B \cap rad(\A)$. Furthermore, it is shown that the subinvariant subalgebras of $\A$ contained in a subalgebra $B$ form a lattice under algebra generation and intersection, the latter of which is obvious. This lattice contains a unique maximal element which is an ideal in $B$, but not necessarily in $\A$. In particular, the collection of all nilpotent subinvariant subalgebras of $\A$ contained in $B$ is a sublattice with a unique maximal element. Finally in the last section we remark that under the assumption $Z^{\ell}(\A) = \{0\}$, the tower theorem for $\A$ follows trivially from the tower theorem for Lie algebra in \cite{SE}.

\bigskip
\section{General Results}
In this section we give some useful basic results on properties of subinvariant subalgebras of a Leibniz algebra $\A$ over field $F$. For a pair of subsets $B$ and $C$ of $\A$, we denote $f(B,C)$ to be the subalgebra of $\A$ containing $BC + CB$. We will denote $f(B,C) = BC + CB = f^1(B, C)$. For all $n \in \mathbb{Z}_{\geq 1}$, we define $f^{n+1}(B, C) = f(B, f^n(B, C))$ inductively. The following result on this function $f$ will be useful later.

\begin{lemma}\label{function} 
Let $B$, $C$ be any two subsets of the Leibniz algebra $\A$. Then $f(B^n, C) \subseteq f^n(B, C)$ for all $n \in \mathbb{Z}_{\geq 1}$.

\end{lemma}
\begin{proof} 
We prove by induction on $n$. Suppose $f(B^n, C) \subseteq f^n(B, C)$. Then $f(B^{n+1}, C) = B^{n+1}C + CB^{n+1} = (BB^n)C + C(BB^n) \subseteq B(B^nC) + B^n(BC) + (CB)B^n + B(CB^n) \subseteq Bf(B^n, C) + f(B^n, BC +CB) \subseteq Bf^n(B, C) + f^n(B, BC+CB) \subseteq f^{n+1}(B, C) + f^{n+1}(B, C) = f^{n+1}(B, C)$
\end{proof}

The following theorem is a Leibniz algebra analog of \cite[Theorem 1]{SE} and can be proved similarly.

\begin{thm}\label{basic}
Let $B, B'$ be subinvariant subalgebras of the Leibniz algebra $\A$ over field $F$, $C$ be a subinvariant subalgebra of $B$ and $K$ any subalgebra of $\A$.
Let $\phi : \A \longrightarrow \A'$ be a homomorphism from $\A$ to a Leibniz algebra $\A'$ over field $F$. Then the following statements hold.
\begin{enumerate}
\item $C$ is subinvariant in $\A$.
\item $B \cap B'$ is subinvariant in $\A$.
\item $K\cap B$ is subinvariant in $K$.
\item $\phi (B)$ is subinvariant in $\phi(\A)$.
\end{enumerate}

\end{thm}

For a subalgebra $B$ of the Leibniz algebra $\A$, define $B^{\omega} = \bigcap_{i=1}^{\infty}B^{i}$ which is an ideal of $B$. 
The following result has been proved in \cite[Theorem 4.4]{B} using a different approach. As shown below it follows immediately
using the function $f$.

\begin{thm}\label{omega}
If $B$ is a subinvariant subalgebra of $\A$, then $B^{\omega}$ is an ideal of $\A$
\end{thm}

\begin{proof} Since $\A$ is finite dimensional, so is $B$. Hence $B^{\omega} =B^{s}$ for some $s \in \mathbb{Z}_{\geq 1}$ in the lower central series.
Also since $B$ is subinvariant in $\A$ we have a normal series $B = B_t \vartriangleleft B_{t-1} \vartriangleleft \cdots \vartriangleleft B_1 \vartriangleleft B_0 =A$. So we have a normal series of length $s+t$:
$$
B^{\omega}=B_{s+t-1} \vartriangleleft B_{s+t-2} \vartriangleleft \cdots \vartriangleleft B_t \vartriangleleft B_{t-1} \vartriangleleft \cdots \vartriangleleft B_0 =A.
$$
where we denote $B^j = B_{j+t-1}, 1\leq j \leq s$. 
Observe that $f(B, \A) = B\A +\A B \subset B_1\A +\A B_1 \subseteq B_1$ since $B_1 \vartriangleleft \A$. Assume that $f^i (B, \A) \subseteq B_i$. Then 
$f^{i+1}(B, \A) = f(B, f^i(B, \A)) \subseteq f(B, B_i) = BB_i + B_iB$. For $0 \leq i \leq t-1$, $BB_i + B_iB \subseteq B_{i+1}B_i + B_iB_{i+1} \subseteq B_{i+1}$ since $B_{i+1} \vartriangleleft B_i$. For $t \leq i \leq s+t-2$, we have $BB_i + B_iB = B_{i+1}$ by definition of lower central series. Thus we have $f^{s+t-1}(B, \A) \subseteq B_{s+t-1} = B^{\omega}$.
Now since $B^{\omega} =B^{s} = B_{s+t-1}$, by Lemma \ref{function} we obtain $B^{\omega}\A + \A B^{\omega} = f(B^{\omega}, \A) = f(B^s, \A) 
= f(B^{s+t-1}, \A) \subset f^{s+t-1}(B, \A) \subseteq B^{\omega}$ implying that $B^{\omega}  \vartriangleleft \A$.

\end{proof}

\begin{cor}
If $B$ is a subinvariant subalgebra of $\A$ and $B^2 = B$, then $B$ is an ideal of $\A$.
\end{cor}
\begin{proof}
This follows from Theorem \ref{omega} since  $B^2 = B$ implies that $B^{\omega} = B$. 
\end{proof}

We recall that a subalgebra $H$ of a Leibniz algebra $\A$ is a Cartan subalgebra if $H$ is nilpotent and 
$N_{\A}(H) = \{x \in \A \mid xh, hx \in H \ \ {\rm for} \ \ {\rm all} \ \ h \in H\} = H$. The existence of a Cartan subalgebra of a finite dimensional Leibniz algebra $\A$ is shown in \cite{B}.

\begin{thm}\label{cartan}
Let $\A$ be a finite dimensional Leibniz algebra over field $F$. Then $\A = \A^{\omega} + H$ where $H$ is a Cartan subalgebra of $\A$.
\end{thm}
\begin{proof}
Let $H$ be a Cartan subalgebra of $\A$. Then $(H+\A^{\omega})/\A^{\omega}$ is a Cartan subalgebra of the nilpotent Leibniz algebra $\A/\A^{\omega}$.
Since a proper subalgebra of a nilpotent Leibniz algebra is contained in its normalizer, we have $(H+\A^{\omega})/\A^{\omega} = \A/\A^{\omega}$. This implies that 
$H+ \A^{\omega} = \A$.
\end{proof}

The following result is an immediate consequence of Theorems \ref{omega} and \ref{cartan}.

\begin{cor}
Let $B$ be a subinvariant subalgebra of the finite dimensional Leibniz algebra $\A$. Then $B = B^{\omega} + H$ where $B^{\omega}$ is an ideal of $\A$ and 
$H$ is a Cartan subalgebra of $B$.
\end{cor}

The following is a Leibniz algebra analog of a well known Lie algebra result which we need to prove in the next theorem.
\begin{lemma}\label{ideal}
Let $\A$ be a nilpotent finite dimensional Leibniz algebra with center $Z(\A)$ and let $B$ be a nonzero ideal of $\A$. Then $B \cap Z(\A) \neq \{0\}$.
\end{lemma}
\begin{proof}
Since $\A$ is nilpotent its center $Z(\A) \neq \{0\}$. Also
$\A$ acts nilpotently on the ideal $B$ by both the left and right multiplications. Hence by Engle's theorem there is a  nonzero element $b \in B$ such that 
$ab = 0 =ba$ for all $a \in \A$. Hence $b \in B\cap Z(\A)$ and $B\cap Z(\A) \neq \{0\}$.
\end{proof}
\begin{thm}\label{center}
Let $\A$ be a finite dimensional Leibniz algebra. If the centralizer $Z_{\A}(\A^{\omega})$ is not contained in $\A^{\omega}$, then the center 
$Z(\A) \neq \{0\}$.  
\end{thm}
\begin{proof}
By Theorem \ref{cartan}, $\A = \A^{\omega} + H$ where $H$ is a Cartan subalgebra of $\A$. By definition, 
$C = Z_{\A}(\A^{\omega})$ is an ideal of $\A$. Consider the subalgebra $\A_1 = C + H$ of $\A$. Then by Theorem \ref{cartan} $\A_1 = \A_1^{\omega} +H_1$
where $ \A_1^{\omega}$ is an ideal of $\A$ and $H_1$ is a Cartan subalgebra of $\A_1$. Also $\A_1^{\omega} = (C+H)^{\omega} \subseteq C$ since $H$ is nilpotent.
\smallskip

If $C \subseteq \A_1^{\omega}$ then  $C \subseteq \A^{\omega}$ which is not the case by assumption. Hence $C \subsetneq \A_1^{\omega}$. Choose 
$c \in C \setminus \A_1^{\omega}$. Since $C \subset \A_1 =  \A_1^{\omega} +H_1$, $c = a + h$ for some $a \in \A_1^{\omega}, h \in H_1$. If $h = 0$ then 
$c = a \in \A_1^{\omega}$ which is a contradiction. Hence $0 \not= h = a -c \in C$ implying that $C \cap H_1 \not= \{0\}$.
\smallskip

Since $C$ is an ideal of $\A$, $C\cap H_1$ is an ideal of $H_1$. The center $Z_1 = Z(H_1) \neq \{0\}$ since $H_1$ is nilpotent. So by Lemma \ref{ideal}, we have $(C\cap H_1)\cap Z_1 \neq \{0\}$. Since $(C\cap H_1)\cap Z_1 \subseteq C\cap H_1 = Z_{\A}(\A^{\omega}) \cap Z(H_1)$, we have 
$Z_{\A}(\A^{\omega}) \cap Z(H_1) \neq \{0\}$. Furthermore, $\A = \A^{\omega} + H \subseteq \A^{\omega} +\A_1 = \A^{\omega} +\A_1^{\omega} +H_1 = \A^{\omega} +H_1$. Hence $Z_{\A}(\A^{\omega}) \cap Z(H_1) \subset Z(\A)$ and $Z(\A)  \neq \{0\}$. 
\end{proof}
\begin{thm}\label{sub-omega}
Let $B$ be a subinvariant subalgebra of $\A$. If $Z_{\A}(B) = \{0\}$, then $Z_{\A}(B^{\omega}) \subseteq B^{\omega}$.
\end{thm}
\begin{proof}
Since $Z_{\A}(B) = \{0\}$ implies $Z(B) = \{0\}$, by Theorem \ref{center} $Z_B(B^{\omega}) \subseteq B^{\omega}$. Suppose $Z_{\A}(B^{\omega}) \subsetneq B^{\omega}$. Then $Z_{\A}(B^{\omega}) \subsetneq B$ since $Z_{\A}(B^{\omega}) \subseteq B$ implies $Z_{\A}(B^{\omega}) = Z_{\A}(B^{\omega})\cap B = Z_B(B^{\omega})\subseteq B^{\omega}$. Since $B^{\omega}$ is an ideal of $\A$, and $Z_{\A}(B)$ is an ideal of 
$\A$ we have $K = Z_{\A}(B^{\omega}) +B$  a subalgebra of $\A$ properly containing $B$. Since $B$ is a subinvariant subalgebra of $\A$, $B$ is a subinvariant subalgebra of $K$. Therefore, $B$ is an ideal of some subinvariant subalgebra $B_1$ of $K$ and $B^{\omega} \subseteq B_1^{\omega}$. Since  $Z_{\A}(B^{\omega}) \subsetneq B^{\omega}$, there exists $z \in  B_1\cap Z_{\A}(B^{\omega})$, and $z \not\in B$. There are four possibilities: Case $(1)$: $z^2 = 0$; Case $(2)$: $z^2 \not\in B$; Case $(3)$: $z^2 \in B\setminus B^{\omega}$; or  Case $(4)$: $z^2 \in B^{\omega}$.

Case $(1)$: If $z^2 = 0$, then  choose $C = {\rm span}\{z, B\}$ which is a subalgebra properly containing $B$ since $z \in B_1$ and  $B \vartriangleleft B_1$. 
Clearly $B^{\omega} \subseteq  C^{\omega}$. Since $z^2 = 0$, $z \in B_1$, and $B \vartriangleleft B_1$, we have $(z + B)(z + B) \subseteq B^2 + (B \cap Z_{\A}(B^{\omega})) = B^2 + Z_{B}(B^{\omega}) \subseteq B^2 + B^{\omega} \subseteq B^2$ which implies that $C^2 = B^2$. Now  $(z + B)(z + B)(z + B) \subseteq (z + B) B^2 \subseteq B^3$ since as shown above $zB \subseteq B^2$ and  by Leibniz identity $zB^2 \subseteq  (zB)B + B(zB) \subseteq B^2B + BB^2 \subseteq B^3$. This implies that $C^3 = B^3$. Since $C^{\omega} = C^t$ for some $t \in \mathbb{Z}_{\geq 2}$, repeating this process we get $C^{\omega} = C^t = B^t = B^{\omega}$. Since $z \in C \cap Z_{\A}(B^{\omega}) = Z_C(C^{\omega})$ and $z \not\in C^{\omega} = B^{\omega} \subset B$, by Theorem \ref{center} we have $Z(C) \neq \{0\}$. Since $Z_{\A}(B) \supseteq Z_C(B) \supseteq Z_C(C) = Z(C)$, we get 
$Z_{\A}(B) \neq \{0\}$ which is a contradiction.

Case $(2)$: If $z^2 \not\in B$, then $z^4 = (z^2)^2 = 0$ since $z^2 \in {\rm Leib}(\A)$ which is an abelian ideal of $\A$. Since $z^2 \in B_1$ and $B \vartriangleleft B_1$, we have $C = {\rm span}\{z^2, B\}$ a subalgebra. Clearly $B^{\omega} \subseteq  C^{\omega}$. Then as in Case $(1)$, we get $B^{\omega} = C^{\omega}$ which leads to a contradiction.

Case $(3)$: In this case $z^2 \in B$, but $z^2 \not\in B^{\omega}$. Since $z \in Z_{\A}(B^{\omega})$ which is an ideal in $\A$, $z^2 \in B \cap  Z_{\A}(B^{\omega}) = Z_B(B^{\omega}) \subseteq B^{\omega}$ which is a contradiction.

Case $(4)$: Suppose $z^2 \in B^{\omega}$. As in case $(1)$ consider the subalgebra $C = {\rm span}\{z, B\}$ properly containing $B$. Since $B^{\omega} \vartriangleleft B$, $(z + B)(z + B) \subseteq B^2 + B^{\omega} \subseteq B^2$. This implies $C^2 = B^2$. Repeating this process as in case $(1)$, we get $C^{\omega} = B^{\omega}$ which leads to a contradiction. 

Hence we have $Z_{\A}(B^{\omega}) \subseteq B^{\omega}$ proving the theorem.
\end{proof}
\bigskip
\section{Characteristic zero case}
In this section all Leibniz algebras are over the field $F$ of characteristic zero. We extend some classic subinvariance results in Lie algebras to Leibniz algebras and show that most of them do not hold over field of finite characteristic. Let $\mathfrak{A}$ be an associative algebra over field $F$. An element $a \in \mathfrak{A}$ is nilpotent if $a^k = 0$ for some positive integer $k$. A subalgebra $\mathfrak{B}$ of the associative Lie algebra $\mathfrak{A}$ is nilpotent if there exist a positive integer $k$ such that every product of $k$ elements in $\mathfrak{B}$ is zero. The radical of $\mathfrak{A}$  denoted by $rad(\mathfrak{A})$ is the maximal nilpotent ideal of $\mathfrak{A}$. For a Lie algebra $L$ of linear transformations  on a finite dimensional vector space over field $F$, we denote $L^*$ to be the associative envelope of $L$. We recall the following interesting result from \cite[page 45]{Jacobson}.

\begin{thm}\cite[Corollary 2.5.2]{Jacobson}\label{associative}
Let $L$ be a Lie algebra of linear transformations on a finite dimensional vector space over field $F$ of characteristic zero. Then $L\cap rad(L^*)$ is the set of all nilpotent elements of $rad(L)$ and  $[rad(L), L] \subseteq rad(L^*)$.
\end{thm}

This leads to the following interesting result for Leibniz algebras over field of characteristic zero.

\begin{thm}\label{radical}
Let $\A$ be a Leibniz algebra over field $F$ of characteristic zero. Then $\A R + R\A \subseteq N$ where $R = rad(\A)$ and $N = nilrad(\A)$.
\end{thm}
\begin{proof}
Consider the homomorphism $\phi : \A \longrightarrow L(\A)$ defined by $\phi (x) = L_x$ for all $x \in \A$. Then $\phi (R) \subseteq rad(L(\A))$. So by Theorem \ref{associative} we have $[\phi (R) , L(\A)] \subseteq [rad(L(\A) , L(\A)] \subseteq rad((L(\A))^*)$. Hence each element in 
$[\phi (R) , L(\A)]$ is nilpotent in the associative envelope $L(\A)^*$. This implies that there exists positive integer $n$ such that for all $n$-tuple of elements $\{(y_i, x_i) \mid 1 \leq i \leq n\}$ in $R\times \A$ or $\A \times R$, we have $[L_{y_n}, L_{x_n}] \cdots  [L_{y_1}, L_{x_1}] = 0$  which implies $(y_nx_n) \cdots (y_1x_1)x = 0$ for all $x \in \A$. Therefore, we have $(\A R + R\A)^{n+1} = \{0\}$ which implies $\A R + R\A \subseteq N$.

\end{proof}

\begin{cor}\label{nilpotent}  Let $\A$ be a solvable Leibniz algebra over field $F$ of characteristic zero. Then $\A^2$ is nilpotent.
\end{cor}
\begin{proof} Since $\A$ is solvable, $R = rad(\A) = \A$. Hence by Theorem \ref{radical}, $\A^2 \subseteq nilrad(\A)$. Hence $\A^2$ is nilpotent.

\end{proof}

\begin{cor} Let $\A$ be a Leibniz algebra over field $F$ of characteristic zero. Then $\A^2 \cap R = \A R + R\A$ where $R = rad(\A)$.
\end{cor}

\begin{proof}
By Levi's Theorem \cite[Theorem 1]{B1}, $\A = R + S$ where $S$ is a semisimple subalgebra of $\A$ and $S\cap R = \{0\}$. Then $\A^2 = (R + S)(R + S) =\A R + R\A + S^2$. Hence $\A^2 \cap R = \A R + R\A$ since $S^2 \cap R \subseteq S\cap R = \{0\}$.
\end{proof}

\begin{cor}\label{nilrad}  Let $\A$ be a Leibniz algebra over field $F$ of characteristic zero. Then $N = nilrad(R)$ where $R = rad(\A)$, and $N = nilrad(\A)$.
\end{cor}

\begin{proof} Since $N$ is a nilpotent ideal of $R$, $N \subseteq nilrad(R)$. By Theorem \ref{radical} $\A (nilrad(R)) + (nilrad(R))\A \subseteq \A R + R\A \subseteq N$ implies that $nilrad(R)$ is a nilpotent ideal of $\A$. Hence $nilrad(R) \subseteq N$ giving $N = nilrad(R)$.

\end{proof}

\begin{thm}\label{solv-sub}
 Let $\A$ be a Leibniz algebra over field $F$ of characteristic zero and $S$ be a solvable subinvariant subalgebra of $\A$. Then $S \subseteq R$ where $R = rad(\A)$. 

\end{thm}
\begin{proof}
By Levi's Theorem \cite[Theorem 1]{B1}, $\A/R$ is semisimple. Since ${\rm Leib}(\A) \subseteq R$, $\A/R$ is a Lie algebra and since $S$ is a subinvariant subalgebra of $\A$, $(S+R)/R$ is a subinvariant subalgebra of $\A/R$. Hence by \cite[Lemma 4.2]{SE}, $(S+R)/R \subseteq rad(\A/R) = \{0\}$ which implies $S \subseteq R$. 
\end{proof}

The following example is a particular case of the example due to Jacobson \cite[page 75]{Jacobson} and Seligman \cite[page 163]{S}, (also see \cite[Example 2.4]{Kristen}) for Lie algebras, hence Leibniz algebras over field $F$ of characteristic $p>2$ and shows that the above theorem does not hold over field of finite characteristic.

\begin{ex}\label{char-p}
Let $F$ be a field of characteristic $p>2$. Let $R$ be the commutative associative algebra over $F$ with basis $\{1, a, a^2, \cdots , a^{p-1}\}$ where $a^p = 0$. Then dim$(R) = p$. The maximal nilpotent ideal $N$ of $R$ has basis $\{a, a^2, \cdots , a^{p-1}\}$ and dim$(R/N) = 1$. Consider the simple Lie algebra $L = s\ell(2, F)$. Then $M = L\otimes R$ is a Lie algebra with $[x\otimes a^s, y\otimes a^t] = [x,y]\otimes a^{s+t}$ for all $x,y \in L$ and $0 \leq s, t \leq p-1$. Since $(L\otimes N)^p = L^p\otimes N^p = L\otimes 0 = 0$, $L\otimes N$ is a nilpotent (hence solvable) ideal of $M$. Since $M/(L\otimes N) = (L\otimes R)/(L\otimes N)$ is isomorphic to $L$, $L\otimes N$ is both the radical and nilradical of $M$.  Consider the derivation $\delta = 1\otimes \frac{d}{da}$ in $Der(M)$ defined by $\delta (x\otimes a^t) =t(x\otimes a^{t-1})$ for all $x\in L$ and $1\leq t\leq p-1$.
Consider the semidirect product Lie algebra (hence Leibniz algebra) $P = M \oplus {\rm span}\{\delta\}$ with the multiplication $[x+\delta, y+\delta] = [x, y] + \delta (y) - \delta (x)$ for all $x, y \in M$. The only nontrivial ideal of $P$ is $M$ which is not solvable. Hence 
$rad(P) (= nilrad(P)) = \{0\}$. Since $L\otimes N \vartriangleleft M \vartriangleleft P$, $L\otimes N$ is a nilpotent, hence solvable subinvariant subalgebra of $P$. However, $L\otimes N \subsetneq rad(P) = \{0\}$. This shows that Theorem \ref{solv-sub} as well as \cite[Lemma 4.2]{SE} does not hold over field of finite characteristic.
\end{ex}

\begin{thm}\label{ideal-nil}
 Let $\A$ be a Leibniz algebra over field $F$ of characteristic zero, $B \vartriangleleft \A$ , $R = rad(B)$, and $N = nilrad(B)$.
 Then $R$ and $N$ are ideals of $\A$ and $\A R + R\A \subseteq N$.
\end{thm}

\begin{proof}
If $B = \A$, then result holds by Theorem \ref{ideal}. So assume $B \subsetneq \A$ and choose an element $x \in \A \setminus B$. Let $C = <x>$ be the cyclic Leibniz algebra generated by $x$. Consider the Leibniz algebra $\bar{B} = B + C$ and denote $\bar{R} = rad (\bar{B})$. Since $B$ is an ideal of $\A$, it is also an ideal of $\bar{B}$. Hence $R$ is subinvariant in $\bar{B}$ and $R \subseteq \bar{R}$ by Theorem \ref{solv-sub}. This implies $R \subseteq \bar{R} \cap B$. Since  $\bar{R}$ and $B$ are ideals of $\bar{B}$, we get $ \bar{R} \cap B \subseteq R$. Hence $R = \bar{R} \cap B$ is an ideal of $\bar{B}$ implying that $xR + Rx \subseteq R$ for all $x \in \A \setminus B$. Since $ R \vartriangleleft B$ we get $R$ is an ideal of $\A$.

Since $R$ and $C$ are solvable,  the subalgebra $R_1 = R + C$ is solvable and $R_1^2$ is nilpotent by Corollary \ref{nilpotent}. Furthermore, $R_1R + RR_1 \subseteq R_1^2 \cap R$ and $R_1^2 \cap R$ is a nilpotent ideal of $R$. Hence $R_1R + RR_1 \subseteq nilrad(R)$. Since $nilrad(R) = nilrad(B)$ by Corollary \ref{nilrad} we have $xR + Rx \subseteq nilrad(R)$ for all $x \in \A$. Therefore, $N = nilrad(B) = nilrad(R)$ is an ideal of $\A$ and $\A R + R\A \subseteq N$. 
\end{proof}

\begin{thm}\label{nil-sub} 
Let $\A$ be a Leibniz algebra over field $F$ of characteristic zero and $M$ be a nilpotent subinvariant subalgebra of $\A$. Then $M \subseteq nilrad(\A)$.
\end{thm}

\begin{proof}
Since $M$ is a subinvariant subalgebra of $\A$, we have a chain of subalgebras $M=M_t \subset M_{t-1} \subset \cdots \subset M_1 \subset M_0 = \A$ such that $M_i \vartriangleleft M_{i-1}, 1 \leq i \leq t$. Denote $N_i = nilrad(M_{i-1}), 1 \leq i \leq t$ which is a nilpotent ideal of $M_{i-1}$. Since $M = M_t$ is a nilpotent ideal of $M_{t-1}$, $M= M_t \subseteq N_t$. By Theorem \ref{ideal-nil}, $N_t$ is an ideal of $M_{t-2}$. Hence $N_t \subseteq N_{t-1}$ and by Theorem \ref{ideal-nil}, $N_{t-1}$ is an ideal of $M_{t-3}$. Continuing this process we get $M \subseteq N_t \subseteq N_{t-1} \subseteq \cdots \subseteq N_2 \subseteq N_1 = nilrad(M_0) = nilrad(\A)$.
\end{proof}

\begin{rmk} As seen in Example \ref{char-p}, $L\otimes N$ is a nilpotent subinvariant subalgebra of $P$, but $L\otimes N \subsetneq nilrad(P) = \{0\}$. Thus Theorem \ref{nil-sub} does not hold over field of finite characteristic.
\end{rmk}

\begin{thm}
Let $\A$ be a Leibniz algebra over field $F$ of characteristic zero and $B$ be a subinvariant subalgebra of $\A$. Then $rad(B) = B\cap rad(\A)$ and $nilrad(B) = B\cap nilrad(\A)$.
\end{thm}

\begin{proof}
Since $B \cap rad(\A)$ (resp. $B \cap nilrad(\A)$) is a solvable (resp. nilpotent) ideal of $B$, we have $B\cap rad(\A) \subseteq rad(B)$ (resp. $B\cap nilrad(\A) \subseteq nilrad(B)$). Also $rad(B)$ (resp. $nilrad(B)$) is a solvable (resp. nilpotent) subinvariant subalgebra of $\A$ implies by Theorem \ref{solv-sub} (resp. Theorem \ref{nil-sub}) that $rad(B)\subseteq B\cap rad(\A)$ (resp. $nilrad(B) \subseteq B\cap nilrad(\A)$). Hence 
$rad(B) = B\cap rad(\A)$ (resp. $nilrad(B) = B\cap nilrad(\A)$).
\end{proof}

\begin{lemma}\label{nil-subin}
Let $\A$ be a Leibniz algebra over field $F$ of characteristic zero and $B, C$ be two nilpotent subinvariant subalgebras of $\A$. Then the subalgebra $D$ generated by $B$ and $C$ is a nilpotent subinvariant subalgebra of $\A$. 
\end{lemma}

\begin{proof} Since $B$ and $C$ are nilpotent, by Theorem \ref{nil-sub}, they are contained in $nilrad(\A)$, hence $D \subseteq nilrad(\A)$. So by \cite[Theorem 4.16]{DMS}, $D$ is a subinvariant subalgebra of $nilrad(\A)$. Therefore, $D$ is a nilpotent subinvariant subalgebra of $\A$.
\end{proof}

\begin{thm}\label{subin}
Let $\A$ be a Leibniz algebra over field $F$ of characteristic zero and $B, C$ be two subinvariant subalgebras of $\A$. Then the subalgebra $D$ generated by $B$ and $C$ is a subinvariant subalgebra of $\A$. 
\end{thm}

\begin{proof} 
If both $B$ and $C$ are nilpotent then the result follows by Lemma \ref{nil-subin}.
Suppose at least one of the subalgebras $B,  C$ is not nilpotent. Then $B^{\omega} \neq \{0\}$ and/or  $C^{\omega} \neq \{0\}$. Since $B$ and $C$ are subinvariant in $\A$, by Theorem \ref{omega} $B^{\omega}$ and $C^{\omega}$ are ideals of $\A$. Hence the subalgebra $K$ generated by 
$\{B^{\omega}, C^{\omega}\}$ is an ideal of $\A$ and quotient algebras $(B+K)/K$ and $(C+K)/K$ are nilpotent in $\A/K$. Hence by Lemma \ref{nil-subin}, $(D+K)/K$ is subinvariant in $\A/K$. Since $D$ contains $K$, $D$ is subinvariant in $\A$.

\end{proof}

\begin{thm}\label{sub-sub}
Let $\A$ be a Leibniz algebra over field $F$ of characteristic zero and $B$ be a subinvariant subalgebra of $\A$. Let $c \in \A$ and $\bar{B}$ be the smallest subalgebra containing $B$ such that $f(c, \bar{B}) \subseteq \bar{B}$. Then $\bar{B}$ is a subinvariant subalgebra of $\A$.
\end{thm}

\begin{proof}
If $c \in B$ then $\bar{B} = B$ and there is nothing to prove. So assume that $c \nin B$. Recall that $f^0(c, B) = B, f^{i+1}(c, B) = f(c, f^i(c, B))$ for all $i \in \mathbb{Z}_{i \geq 0}$. Consider the subalgebra $C = \sum_{i \geq 0}f^i(c, B)$. By definition  $B = f^0(c, B)  \subseteq C$. Since $cf^i(c, B) + f^i(c, B)c = f(c, f^i(c, B)) = f^{i+1}(c, B)$, we have $c C + Cc \subseteq C$ which implies $\bar{B} \subseteq C$. Also by definition $B \subseteq \bar{B}$ and $f(c, B) \subseteq f(c, \bar{B}) \subseteq \bar{B}$. Assume $f^i(c, B) \subseteq \bar{B}$. Then $f^{i+1}(c, B) = f(c, f^i(c, B)) \subseteq f(c, \bar{B}) \subseteq \bar{B}$. So by induction $f^i(c, B) \subseteq \bar{B}$ for all $i \geq 0$ which implies $C = \bar{B}$.

If $B$ is nilpotent, then $B \subseteq nilrad(\A)$. Since $nilrad(\A)$ is an ideal of $\A$, we have $f(c, B) = cB +Bc \subseteq nilrad(\A)$. Suppose $f^i(c, B) \subseteq nilrad(\A)$. Then $f^{i+1}(c, B) = f(c, f^i(c, B)) = cf^i(c, B) + f^i(c, B)c \subseteq nilrad(\A)$. So by induction $\bar{B} = \sum_{i \geq 0}f^i(c, B) \subseteq nilrad(\A)$. Hence by \cite[Theorem 4.16]{DMS}, $\bar{B}$ is a subinvariant subalgebra in $nilrad(\A)$, hence in $\A$.

Suppose $B$ is not nilpotent.Then by Theorem \ref{omega}, $B^{\omega} \neq \{0\}$ is an ideal of $\A$. Note that $B_1 = B/B^{\omega}$ is a nilpotent subalgebra of $\A_1 = \A/B^{\omega}$. Define the subalgebra $\bar{B}_1 = \bar{B}/B^{\omega}$. Then as before $\bar{B}_1 = \sum_{i \geq 0}((f^i(c, B) + B^{\omega})/B^{\omega} = \sum_{i \geq 0} f^i(c_1, B_1) \subseteq \A_1$ where $c_1 = c+B^{\omega} \in \A_1$. Hence $\bar{B}_1$ is a subinvariant subalgebra of $\A_1$ which implies $\bar{B}$ is a subinvariant subalgebra of $\A$ since $\bar{B} \supseteq B^{\omega}$.
\end{proof}

\begin{cor}\label{minimal}
Let $\A$ be a Leibniz algebra over field $F$ of characteristic zero and $B$ be a subalgebra. Let $\underline{B}$ be the subalgebra of $B$ generated by all subinvariant subalgebras of $\A$ contained in $B$. Then $\underline{B}$ is an ideal of $B$.
\end{cor}

\begin{proof} By definition $\underline{B}$ is a subinvariant subalgebra of $\A$ and $\underline{B} \subseteq B$. Let $c \in B$ and $B_1$ be the minimal subalgebra of $\A$ such that $\underline{B} \subseteq B_1$ and $f(c, B_1) \subseteq B_1$. Then by Theorem \ref{sub-sub}, $B_1$ is a subinvariant subalgebra of $\A$. By definition $B_1 \subseteq B$, hence $B_1 \subseteq \underline{B}$ which implies $B_1 = \underline{B}$. Therefore, for all $c \in B$, we have $c \underline{B} + \underline{B} c = f(c, \underline{B}) = f(c, B_1) \subseteq B_1 = \underline{B}$. Hence 
$\underline{B}$ is an ideal of $B$.
\end{proof}

The following example shows that in Corollary \ref{minimal}, the subalgebra $\underline{B}$ is not necessarily an ideal of $\A$.

\begin{ex} Let $\A = {\rm span}\{x, y, z, t\}$ be  the Leibniz algebra over field $\mathbb{C}$ with nonzero multiplications $tx = x, ty = y, zy = x$. Then $B = {\rm span}\{z, t\}$ is an abelian self-normalizing subalgebra (hence a Cartan subalgebra) of $\A$. The subalgebra ${\rm span}\{z\}$ is a subinvariant subalgebra $\A$ since we have a chain of ideals ${\rm span}\{z\} \vartriangleleft {\rm span}\{z, x\} \vartriangleleft {\rm span}\{z, x, y\} \vartriangleleft \A$. Since $B$ is self-normalizing, $\underline{B} \neq B$. Hence $\underline{B} = {\rm span}\{z\}$ which is clearly an ideal of $B$, but not an ideal of $\A$.

\end{ex}
\section{Remark on Tower Theorem}
In the fundamental work on subinvariance by Wielandt \cite{W} for groups and by Schenkman \cite{SE} for Lie algebras, there is an interesting result known as tower theorem. Let us recall the tower theorem for Lie algebras. Let $\g$ be a finite dimensional Lie algebra over field $F$ of characteristic zero. Assume that $\g$ has center zero. Then $\g$ is isomorphic to the ideal $\{ad_x \mid x \in \g\}$ in $Der(\g)$. Thus there is a natural embedding of $\g \subset Der(\g)$. Define $\g_0 = \g, \g_1 = Der(\g)$ and $\g_i = Der(\g_{i-1})$ for all $i \in \mathbb{Z}_{\geq 1}$. It is shown in \cite[Theorem 17]{SE} that there is a natural embedding of $\g_{i-1}$ in $\g_i$ and the tower of Lie algebras $\{\g = \g_0 \subset \g_1 \subset \g_2 \subset \cdots \}$ terminates in a limit Lie algebra $\g_{\infty}$ which is a complete Lie algebra and ${\rm dim}(\g_{\infty}) \leq (d+c)$, where $d = {\rm dim}(Der(\g^{\omega}))$ and $c$ is the dimension of the center of $\g^{\omega}$. Later Carles \cite{C} gave a description of each member $\g_i$ in the tower including the limit $\g_{\infty}$ and extended the tower theorem for Lie algebras with nontrivial center.

 Now let us consider the case of a Leibniz algebra $\A$ over field $F$ of characteristic zero. Assume that the left center $Z^{\ell}(\A) = \{0\}$. Then 
 $\A$ is isomorphic to $L(\A) = \{L_x \mid x \in \A\} \subset Der(\A)$. Thus we have a natural embedding of $\A$ in $Der(\A)$. Now define $\A_0 = \A, \A_1 = Der(\A)$ and $\A_i = Der(\A_{i-1}$ for all $i \in \mathbb{Z}_{\geq 1}$. Since $Leib(\A) \subseteq Z^{\ell}(\A)$, we have $Leib(\A) = \{0\}$ which implies that $\A$ is indeed a Lie algebra and the tower theorem for Leibniz algebras follows by \cite[Theorem 17]{SE} for Lie algebra. Observe that by definition of complete Leibniz algebra in \cite{Kristen}, since the limit Leibniz algebra $\A_{\infty}$ is complete as a Lie algebra, it will be also complete as a Leibniz algebra. It would be an interesting problem to investigate whether there is any Leibniz algebra analog of Carles' \cite{C} Lie algebra results when $Z^{\ell}(\A) \neq \{0\}$ and $Leib(\A)\neq \{0\}$.

\end{document}